\documentclass[12pt]{amsart}

\usepackage{cite}
\usepackage{amsfonts,xcolor}
\usepackage{amssymb, latexsym, amsthm, amsmath, verbatim}
\usepackage{hyperref}

\newtheorem{theorem}{Theorem}[section]

\newtheorem{prop}[theorem]{Proposition}
\newtheorem{lemma}[theorem]{Lemma}
\newtheorem{cor}[theorem]{Corollary}

\makeatletter
\def\imod#1{\allowbreak\mkern5mu{\operator@font mod}\,\,#1}
\makeatother

\allowdisplaybreaks[4]

\begin{document}

\title[Hilbert Poincar\'e Series and kernels for products of $L$-functions]{Hilbert Poincar\'e Series and kernels for products of $L$-functions}
\author{Mingkuan Zhang and Yichao Zhang}
\address{School of Mathematics, Harbin Institute of Technology, Harbin 150001, P. R. China}
\email{19b912043@stu.hit.edu.cn}
\address{MOE-LCSM, School of Mathematics and Statistics, Hunan Normal University, Changsha, Hunan 410081, P. R. China}
\email{yichao.zhang@hunnu.edu.cn}
\date{}
\thanks{} 

\begin{abstract} We study Hilbert Poincar\'e series associated to general seed functions and construct Cohen's kernels and double Eisenstein series as series of Hilbert Poincar\'e series. Then we calculate the Rankin-Cohen brackets of Hilbert Poincar\'e series and Hilbert modular forms and extend Zagier's kernel formula to totally real number fields. Finally, we show that the Rankin-Cohen brackets of two different types of Eisenstein series are special values of double Eisenstein series up to a constant.
\end{abstract}

\maketitle
%\tableofcontents

\section{Introduction}

Zagier \cite{Z77} proved a formula, known as Zagier's kernel formula, that relates the Petersson inner product of the Rankin-Cohen bracket of a modular form $f$ and an Eisenstein series against a cusp form $g$ to the special value of the Rankin-Selberg convolution of $f$ and $g$, following a method of Rankin in \cite{R52} and generalizing a formula therein. Such a formula is important in Shimura-Eichler-Manin theory, and was used to prove the algebraicity of critical values of $L$-functions \cite{KZ84}. In order to generalize Zagier's kernel formula and reach general $L$-values, Diamantis and O'Sullivan \cite{DO13} introduced double Eisenstein series as kernels for products of $L$-functions. They showed that certain values of the double Eisenstein series are the Rankin-Cohen brackets of two Eisenstein series and gave a new proof of Manin's periods theorem \cite{M73}. More recently, Choie and Zhang \cite{CZ20} extended their results to full level Hilbert modular forms over real quadratic number fields and recovered Shimura's arithmetic result on Hilbert $L$-functions in \cite{S78} when the narrow class number is 1. Note that Rankin's method has also been extended to Jacobi forms by Choie and Kohnen \cite{CK97} and Jacobi forms of several variables by Ramakrishnan and Sahu \cite{RS10}. 

In this paper, for any totally real number field $F$, we first construct Hilbert Poincar\'e series $\mathbb{P}_{\mathbf{k}}(\phi ; z)$ associated to general $\Gamma_\infty$-invariant seed functions $\phi(z)$ of weight $\mathbf{k} \in \mathbb{Z}_{>2}^d$ with $d=\mathrm{deg}\ F$. It defines a Hilbert modular form of weight $\mathbf{k}$ when the Fourier coefficients of $\phi$ grow moderately, and specializes to the usual Hilbert Poincar\'e series $P_{\mu}(z;\mathbf{k},\mathfrak{c},\mathfrak{n})$ and Hilbert Eisenstein series $E_\mathbf{k}(z;\mathfrak{c},\mathfrak{n})$, see Section 3 for more details. Then we prove an identity in Section 4 for the Rankin-Cohen brackets of $P_{\mu}(z;\mathbf{k},\mathfrak{c},\mathfrak{n})$ and a Hilbert modular form $f$ of weight $\mathfrak{l}$, which takes the form $[f, P_{\mu}]_{\mathbf{n}}=\mathbb{P}_{\mathbf{k}+\mathfrak{l}+2\mathbf{n}}(\phi)$ for a concretely constructed seed function $\phi$. A similar formula also holds for Hilbert Eisenstein series and extends the corresponding formula in \cite{CKR07} to arbitrary level. See Theorem \ref{poincare series rankin cohen} and its corollaries for the precise statements. Similar treatment in the elliptic case can be found in \cite{W18}.
 
In Section 5, we construct Cohen's kernel $\mathcal{C}_k(z;s)$ as a series of Poincar\'e series in the parallel weight case, which is absolutely convergent for $Re(s)<-\frac{1}{4}$. By employing Hecke operators on Hilbert Poincar\'e series, we express the double Eisenstein series $E_{s, k-s}(\cdot , w)$ as a series of Poincar\'e series and then prove Zagier's kernel formula for $E_\mathbf{k}(z;\mathfrak{c},\mathfrak{n})$ and another Eisenstein series constructed by Shimura \cite{S78} in Theorem \ref{Zagier's kernel formula}. Finally, we recover Shimura's algebraicity result of Hilbert $L$-functions \cite{S78} in the parallel weight case with trivial central character in Corollary \ref{rationality}.

%\section*{Acknowledgement}
%To add a non numbered section you just use \section*{Acknowledgement}
%If you want to insert this section on the table of contents, use \addcontentsline{toc}{section}{Acknowledgement}

\section{Preliminaries}
In this section, we briefly introduce the basics on Hilbert modular forms and set up the related notations.

Let $F$ be a totally real number field of degree $d>1$ over $\mathbb{Q}$, with ring of integers $\mathcal{O}=\mathcal{O}_{F}$, different $\mathfrak{d}$, discriminant $D$, distinct real embeddings $\sigma_{1}, \ldots, \sigma_{d}$, and narrow class number $h^+$. For an element $a\in F$, set $a_i=\sigma_i(a)$, the $i$-th component of $a$ under the embedding $F\subset F_\mathbb{R}=F\otimes_\mathbb{Q}\mathbb{R}$. For a subset $S\subset F$, set $S^+$ to be the subset of totally positive elements of $S$, that is,
\[S^+=\{a\in S: a_i>0, i=1,\ldots, d\}.\]
In particular, $\mathcal{O}^{\times +}$ is the set of totally positive units of $\mathcal{O}$. For an abelian group $S$, we denote $S^*=S\backslash\{ 0 \}$. We shall reserve $A=A(F)>1$ for a constant such that each class of $F^\times/\mathcal{O}^{\times+}$ contains a representative $\mu$ such that $A^{-1}N(\mu)^{\frac{1}{d}}\leq |\mu_i|\leq AN(\mu)^{\frac{1}{d}}$ for each $i$.

Let $\mathfrak{c}$ be a fractional ideal, $\mathfrak{n}$ be a nonzero integral ideal of $F$, and $\Gamma_0(\mathfrak{c}, \mathfrak{n})$ be the congruence subgroup of level $\mathfrak{n}$ defined by
$$\Gamma_0(\mathfrak{c}, \mathfrak{n})=\left\{\begin{pmatrix}
a&b\\c&d
\end{pmatrix}: a,d\in\mathcal{O}, b\in (\mathfrak{cd})^{-1}, c\in\mathfrak{cnd}, ad-bc\in\mathcal{O}^{\times +}\right\}.$$
We breifly denote $\Gamma=\Gamma_0(\mathfrak{c}, \mathfrak{n})$.
It is well-known that $\Gamma$ acts discontinuously on $\mathbb{H}^{d}$  with finite covolume via the componentwise M\"obius transformation. For $\mathbf{k}=(k_1,\dots,k_d)$, denote by $M_{\mathbf{k}}\left(\Gamma\right)$ the space of Hilbert modular forms of weight $\mathbf{k}$, level $\Gamma$ with trivial character, i.e. the space of holomorphic functions $f(z)$ on $\mathbb{H}^{d}$ such that for $\gamma=\begin{pmatrix}
a&b\\c&d
\end{pmatrix} \in \Gamma, f(\gamma(z))= \mathrm{det}(\gamma)^{-\frac{\mathbf{k}}{2}}N(c z+d)^{\mathbf{k}} f(z)$. Here for $z=\left(z_{1}, \ldots, z_{d}\right) \in \mathbb{H}^{d}$,
$$\mathrm{det}(\gamma)=(\mathrm{det}(\gamma_1),\dots,\mathrm{det}(\gamma_d)) ,\quad N(c z+d)^{\mathbf{k}}=\prod_{i=1}^{d}\left(c_i z_{i}+d_i\right)^{k_i}.$$
According to Koecher's principle \cite{F90}, each $f$ in $M_{\mathbf{k}}\left(\Gamma\right)$ has the following Fourier expansion at $\infty$
$$f(z)=\sum_{\nu \in \mathfrak{c}^{+}\cup\{0\}} a(\nu) \exp (2 \pi i \operatorname{Tr}(\nu z)),$$
where $\operatorname{Tr}(\nu z)=\sum_{i=1}^{d} \nu_i z_{i} $. The action of $\begin{pmatrix}
\varepsilon&0\\0&1
\end{pmatrix}$ for $\varepsilon\in\mathcal{O}^{\times +}$ shows $a\left(\varepsilon \nu\right)=\varepsilon^{\frac{\mathbf{k}}{2}}a(\nu)$. The Fourier expansion of $f$ at other cusps can be defined and if the constant term $a(0)=0$ at each cusp, $f$ is called a cusp form and the corresponding space is denoted by $S_\mathbf{k}(\Gamma)$. If the weight $\mathbf{k}$ is parallel, we shall also denote $M_{\mathbf{k}}\left(\Gamma\right)$ and $S_{\mathbf{k}}\left(\Gamma\right)$ as $M_k\left(\Gamma\right)$ and $S_k\left(\Gamma\right)$ respectively.  For any $\varepsilon > 0$, we have $a_\mu(f)=O\left(N(\mu)^{\mathbf{k} / 2+\varepsilon}\right)$ for every cusp form $f \in S_\mathbf{k}(\Gamma)$\cite{HP03}. The Petersson inner product on $S_{\mathbf{k}}(\Gamma)$ is defined by
$$\left\langle f, h\right\rangle=\int_{\Gamma \backslash \mathbb{H}^{d}} f(z) \overline{h}(z)N(y)^{\mathbf{k}}d\mu(z),\quad d\mu(z)=\prod_i\frac{d x_{i} d y_{i}}{y_i^{2}}.$$

To introduce Hecke operators, it is necessary to work in the ad\'elic setting. We briefly recall ad\'elic Hilbert modular forms and refer to \cite{S78} for further details. Denote $\hat{\mathcal{O}}=\prod_\mathfrak{p}\mathcal{O}_\mathfrak{p}$, where $\mathfrak p$ runs over all nonzero prime ideals of $\mathcal{O}$. Let $\mathcal{M}_{k}(\mathfrak{n})$ ($\mathcal{S}_{k}(\mathfrak{n})$) be the space of (cuspidal) ad\'elic Hilbert modular forms with trivial central character of parallel weight $k$ with level
$$K_{0}(\mathfrak{n})=\left\{\left(\begin{array}{ll}
a & b \\
c & d
\end{array}\right) \in \mathrm{GL}_{2}(\hat{\mathcal{O}}): c \in \mathfrak{n} \hat{\mathcal{O}}\right\}.$$
Since $K_{0}(\mathfrak{n})$ has determinant $\widehat{\mathcal{O}}^{\times}$, $\mathrm{GL}_{2}(F) \backslash \mathrm{GL}_{2}(\mathbb{A}) / K_{0}(\mathfrak{n}) \mathrm{GL}_{2}^{+}(\mathbb{R})^{d}$ has cardinality $h^{+}$. Let $\{ t_{j} \}_{j =1}^{h^+}$ be finite ideles such that the fractional ideals $\mathfrak{c}_{j}:=t_{j}\mathcal{O}$ form a complete set of representatives of the narrow class group of $F$. Set $\Gamma_{j}=\Gamma_{0}\left(\mathfrak{c}_{j}, \mathfrak{n}\right)$. Shimura proved the following isomorphism between ad\'elic Hilbert modular forms and $h^+$-tuples of classical Hilbert modular forms.

\begin{theorem}[Shimura \cite{S78}]\label{classical and adelic}
For the parallel weight $k$, there exist isomorphisms of complex vector spaces
$$\mathcal{M}_{k}\left(\mathfrak{n}\right) \cong \prod_{j=1}^{h^{+}} M_{k}\left(\Gamma_{j}\right) 
\quad \text { and } \quad 
\mathcal{S}_{k}\left(\mathfrak{n}\right) \cong \prod_{j=1}^{h^{+}} S_{k}\left(\Gamma_{j}\right).$$
\end{theorem}

Under such isomorphisms, we may write an element $f \in \mathcal{S}_k(\mathfrak{n})$ as $f=\left(f_{j}\right)$ with $f_{j} \in$ $S_k\left(\Gamma_{j}\right)$. The Petersson inner product is defined componentwisely as
$\langle f, h\rangle=\sum_{j}\left\langle f_{j}, h_{j}\right\rangle_{\Gamma_{j}}$.
For each integral ideal $\mathfrak{m}$, assuming that $\mathfrak{m}=\mu \mathfrak{c}_j^{-1}$ with $\mu \in F^{+}$. Set
$c(\mathfrak{m}, f)= N(\mathfrak{c}_j)^{-\frac{k}{2}}a_{j}(\mu)$,
where $a_{j}(\mu)$ is the $\mu$-th Fourier coefficient of $f_{j}$. This is well-defined and called the $\mathfrak{m}$-th Fourier coefficient of $f$.

For each nonzero integral ideal $\mathfrak{m}$, there exists a Hecke operator $T_\mathfrak{m}$ on $\mathcal{M}_k (\mathfrak{n})$. The Hecke operators are multiplicative, commute with each other, and away from the level $\mathfrak{n}$, are also self-adjoint and hence normal. The effect of the Hecke operators on Fourier coefficients is given by
$$c\left(\mathfrak{a}, T_{\mathfrak{m}}f\right)=\sum_{\mathfrak{r} \supset \mathfrak{a}+\mathfrak{m}}\chi_0(\mathfrak{r}) N(\mathfrak{r})^{k-1} c\left(\mathfrak{a} \mathfrak{m} \mathfrak{r}^{-2}, f\right),$$
where $\chi_0(\mathfrak{r})=1$ if $\mathfrak{r}$ is coprime to $\mathfrak{n}$ and $\chi_0(\mathfrak{r})=0$ otherwise. 
There exists a basis $\mathcal{H}_k$ of $\mathcal{S}_{k}\left(\mathfrak{n}\right)$ consisting of normalized cuspidal Hecke eigenforms. We call elements in $\mathcal{H}_k$ `primitive forms'. Here $f$ is normalized if the Fourier coefficient $c(\mathcal{O},f) = 1$. Therefore, for $f \in \mathcal{H}_k$, $T_\mathfrak{m} f = c(\mathfrak{m},f) f $ and $c(\mathfrak{m},f)$ is real  \cite[Section 1.15]{G89}.

The $L$-function associated to $f \in \mathcal{M}_k (\mathfrak{n})$ is
$$L(s, f)=\sum_{\mathfrak{n}} c(\mathfrak{n}, f) N(\mathfrak{n})^{-s}, \quad s \in \mathbb{C},$$
where $\mathfrak{n}$ runs over all non-zero integral ideals of $\mathcal{O}$. This function is absolutely convergent for $\operatorname{Re}(s)>1$. The Rankin-Selberg $L$-function of $f \in \mathcal{M}_k (\mathfrak{n})$ and $g \in \mathcal{M}_l (\mathfrak{n})$ is defined to be
$$L(f\times g, s)=\sum_{\mathfrak{n}} c(\mathfrak{n}, f) \overline{c(\mathfrak{n}, g) }N(\mathfrak{n})^{-s} .$$
If $f$ is a cuspidal eigenform and $g$ is an eigenform, then $L(f\times g, s)$ is absolutely convergent for $\operatorname{Re}(s)>\frac{k+l}{2}$ or $\frac{k-1}{2}+l$ depending on whether $g$ is cuspidal or not \cite{B06}.

\section{Hilbert Poincar\'e series}

Let $\mathbf{k} \in \mathbb{Z}_{>2}^d$ and let $\Gamma_{\infty}=U(\Gamma)Z(\Gamma)N(\Gamma)$ be the parabolic subgroup with
\begin{align*}
U(\Gamma)&=\left\{ \left(\begin{array}{c c}
\varepsilon & 0 \\
0 & 1
\end{array}\right)\ |\ \varepsilon\in \mathcal{O}^{\times +} \right\},\\ Z(\Gamma)&=\left\{\begin{pmatrix}
\varepsilon&0\\0&\varepsilon
\end{pmatrix}| \varepsilon \in \mathcal{O}^\times\right\},\\ N(\Gamma)&=\left\{\begin{pmatrix}
1&b\\0&1
\end{pmatrix}| b \in (\mathfrak{cd})^{-1}\right\}.    
\end{align*}
Let $\phi(z)=\sum_{\nu\in\mathfrak{c}^+\cup\{ 0 \} } a_{\nu}(\phi) e(\nu z)$ be a function on $\mathbb{H}^d$ such that its coefficients $a_\nu(\phi)$ grow slowly enough and that $\phi(z)$ is invariant under $\Gamma_\infty$. The Hilbert Poincar\'e series associated to $\phi$ is defined to be
$$\mathbb{P}_{\mathbf{k}}(\phi ; z)=\sum_{M \in \Gamma_{\infty} \backslash \Gamma} \phi |_{\mathbf{k}} M(z)
=\sum_{(c, d)} \sum_{\varepsilon \in \mathcal{O}^{\times +}}\sum_{\nu\in\mathfrak{c}^+/ \mathcal{O}^{\times +}} a_{\nu}(\phi) \det(\gamma)^{\frac{\mathbf{k}}{2}}  N(c z+d)^{-\mathbf{k}} e(\varepsilon \nu \gamma z).$$
It converges absolutely and uniformly on compact subsets of $\mathbb{H}^d$ and defines a Hilbert modular form of weight $\mathbf{k}$ when a moderate growth condition is satisfied.

\begin{prop}\label{growth condition}
A sufficient condition for the series $\mathbb{P}_{\mathbf{k}}(\phi ; z)$ to converge absolutely and locally uniformly is that  $a_\nu=O\left(N(\nu)^{\frac{\mathbf{k}}{2}-2-\varepsilon}\right)$ for some $\varepsilon \in \mathbb{R}_{>0}$ .
\end{prop}
\begin{proof}
We apply the same procedure as in \cite[page 53]{G89}. It suffices to prove the convergence for
$$\sum_{\nu\in\mathfrak{c}_+/\mathcal{O}^{\times +}}a_\nu \int_{(0,+\infty)^d}y^{\frac{\mathbf{k}}{2}-2}|e^{-2\pi tr(\nu y)}|dy,$$
whose convergence is the same as
$$\sum_{\nu\in\mathfrak{c}_+/\mathcal{O}^{\times +}}a_\nu \int_{(0,1)^d}y^{\frac{\mathbf{k}}{2}-2}|e^{-2\pi tr(\nu y)}|dy.$$
Since
$$\left|\sum_{\nu\in\mathfrak{c}_+/\mathcal{O}^{\times +}}a_\nu \int_{(0,1)^d}y^{\frac{\mathbf{k}}{2}-2}|e^{-2\pi tr(\nu y)}|dy\right|
\leq \sum_{\nu\in\mathfrak{c}_+/\mathcal{O}^{\times +}} |a_\nu | \int_{(0,1)^d}y^{\frac{\mathbf{k}}{2}-2}\frac{C}{(2\pi N(\nu y))^{s}}dy$$
for any $s\in\mathbb{R}_{>0}^d$ with a constant $C$ depending on $s$. When $a_\nu=O\left(N(\nu)^{\frac{\mathbf{k}}{2}-2-\varepsilon}\right)$, it's easy to verity the absolute convergence by choosing suitable $\frac{\mathbf{k}}{2}-1-\varepsilon < s < \frac{\mathbf{k}}{2}-1$ such that $s-(\frac{\mathbf{k}}{2}-1-\varepsilon)$ is of parallel weight.
\end{proof}

Now we can see some explicit examples. When $\phi(z)=\sum_{\nu\in \mu\mathcal{O}^{\times +}} (\frac{\nu}{\mu})^{\frac{\mathbf{k}}{2}} e(\nu z)$ for $\mu\in\mathfrak{c}^+$ and $\mathbf{k}\in\mathbb{Z}_{>2}^d$, $\mathbb{P}(\phi)=P_{\mu}(z;\mathbf{k},\mathfrak{c},\mathfrak{n})$ converges absolutely and belongs to $S_\mathbf{k}(\Gamma)$, see \cite{G89} for more details. When $\mathbf{k}=(k,\dots,k)$, for $\nu\in\mathfrak{c}^+$, the $\nu$-th coefficient $c_k(\nu,\mu)=c_k(\nu,\mu;\mathfrak{c},\mathfrak{n})$ of $P_{\mu}(z;k, \mathfrak{c},\mathfrak{n})$ equals 
\begin{align*}
\chi_{\mu}(\nu)+ \left(\frac{N(\nu)}{N(\mu)}\right)^{\frac{k-1}{2}}\frac{(2\pi)^d(-1)^{dk/2}N(\mathfrak{cd})}{D^{1 / 2}}
\sum_{\varepsilon\in \mathcal{O}^{\times +}}\sum_{c\in(\mathfrak{cnd}/\mathcal{O}^{\times})^{\ast}}\frac{ S_{c(\mathfrak{cd})^{-1}}(\nu, \varepsilon\mu  ; c)}{|N(c)|}NJ_{k-1} \left(\frac{4 \pi \sqrt{\nu\varepsilon\mu}}{|c|}\right)
\end{align*}
as in \cite{ZZ}. For any $\varepsilon\in\mathcal{O}^{\times+}$ and $\mu\in\mathfrak{c}^+$, it is easy to see that
$P_{\varepsilon\mu}(z ; k, \mathfrak{c},\mathfrak{n})=P_{\mu}(z ; k, \mathfrak{c},\mathfrak{n})$. Therefore, we may assume $\mu\in\mathfrak{c}^+/\mathcal{O}^{\times+}$ and choose representatives $\mu$ such that $A^{-1}N(\mu)^{\frac{1}{d}}\leq \mu_i\leq AN(\mu)^{\frac{1}{d}}$ for each $i$.

When $\phi(z)=1$, $E_\mathbf{k}(z;\mathfrak{c},\mathfrak{n})=\sum_{M \in \Gamma_\infty \backslash \Gamma}\left(\left.1\right|_{\mathbf{k}} M\right)(\tau)$ is the Hilbert Eisenstein series of weight $\mathbf{k}$ on $\Gamma$ if $\mathbf{k}\in\mathbb{Z}_{>2}^d$ (see, for example, \cite{F90}). When $\mathbf{k}=(k,\dots,k)$, it has the Fourier expansion
$$\zeta_P(k)+\frac{(2\pi)^{dk}}{e^{\pi i dk/2}\Gamma(k)^d \sqrt{D}N(\mathfrak{cd})}
\sum_{\nu\in \mathfrak{c}^+} \sum_{c\in\mathfrak{cnd}/\mathcal{O}^\times} N(c)^{-k}   S_{\mathfrak{c}}(\nu,0;c) N(\nu)^{k-1} e(\nu z),$$
where $\zeta_P$ is the partial zeta function associated to principal integral ideals of $F$. In particular, when the narrow class number is one, we have
$$E_k(z;\mathcal{O},\mathfrak{n})=\sum_{(c,d)\in (\mathfrak{nd})\times\mathcal{O}\slash \mathcal{O}^{\times}} N(cz+d)^{-k}
=\zeta_F(k)+\frac{(2 \pi)^{d k}}{e^{\pi i dk/2} \Gamma(k)^d D^{1 / 2}N(\mathfrak{d})^{k}} \sum_{b\in\mathcal{O}^+}\sigma_{k-1}(b)e(nbz),$$
where $n\in\mathfrak{n}$ is a totally positive generator of $\mathfrak{n}$ and $\sigma_{k-1}(b)=\sum_{(b)\subset\mathfrak{a}}N(\mathfrak{a})^{k-1}$.

More generally, if $\phi$ is a Hilbert modular form of parallel weight $\mathbf{k}$, then expanding formally yields
$$\mathbb{P}_{\mathbf{k}+\mathfrak{l}}(\phi ; \tau) =\sum_{M \in \Gamma_{\infty} \backslash \Gamma}N(c \tau+d)^{\mathbf{-k}-\mathfrak{l}} \phi\left(\frac{a \tau+b}{c \tau+d}\right) 
=\sum_{M \in \Gamma_{\infty} \backslash \Gamma}\frac{(c \tau+d)^{\mathbf{k}}}{N(c \tau+d)^{\mathbf{k}+\mathfrak{l}}} \phi(\tau)
=\phi(\tau) E_{\mathfrak{l}}(\tau)$$
when the parallel weight $\mathfrak{l}$ is sufficiently large compared to the growth of the coefficients of $\phi$.

Given a function $\phi(z)$ with coefficients $c_\mu$, one can also consider the series 
$$\mathbb{P}_k^{\prime}(\phi)=c_0 E_\mathbf{k}(z;\mathfrak{c},\mathfrak{n})+\sum_{\mu \in \mathfrak{c}^+ / \mathcal{O}^{\times +}} c_{\mu}P_{\mu}(z;k,\mathfrak{c},\mathfrak{n})$$
for parallel weight $k$, and the non-parallel weight case can be studied similarly. Since $S_k(\Gamma)$ is finite dimensional, the convergence of $\sum_{\mu \in \mathfrak{c}^+ / \mathcal{O}^{\times +}} c_{\mu}P_{\mu}(z;k,\mathfrak{c},\mathfrak{n})$ to a cusp form is equivalent to the convergence of the series
$$\sum_{\mu \in \mathfrak{c}^+ / \mathcal{O}^{\times +}} c_{\mu} \langle f,P_{\mu}(z;k,\mathfrak{c},\mathfrak{n})\rangle=\sum_{\mu \in \mathfrak{c}^+ / \mathcal{O}^{\times +}} c_{\mu} a_\mu(f) N(\mathfrak{c} \mathfrak{d})^{-1} N(\mu)^{1-k} \sqrt{D}\left((4 \pi)^{1-k} \Gamma(k-1)\right)^d$$
for every cusp form $f \in S_k(\Gamma)$. The upper bound $a_\mu(f)=O\left(N(\mu)^{k / 2+\varepsilon}\right)$ implies that this is absolutely convergent when the bound $c_\mu=O\left(N(\mu)^{\frac{k}{2}-2-\varepsilon}\right)$ holds. A similar discussion for elliptic Poincar\'e series can be found in \cite{W18}, where the convergence conditions for $\mathbb{P}_k^{\prime}(\phi)$ and $\mathbb{P}_k(\phi)$ are different.

\section{Rankin-Cohen bracket of Hilbert modular forms}
In this section, we calculate the Rankin-Cohen brackets of Hilbert Poincar\'e series and Hilbert modular forms. Choie, Kim and Richter defined Rankin-Cohen bracket in \cite{CKR07} for general subgroups of $\mathrm{SL}_2(\mathbb{R})^d$, which we recall as follows. Let $f_r: \mathbb{H}^d \rightarrow \mathbb{C}$ be holomorphic and $\mathbf{k}_r \in \mathbb{N}_0^d$ for $r=1,2$. For all $\mathbf{n}=\left(n_1, \ldots, n_d\right) \in \mathbb{N}_0^d$, the $\mathbf{n}$-th Rankin-Cohen bracket is defined as
$$\left[f_1, f_2\right]_{\mathbf{n}}=\sum_{\substack{\mathfrak{l}=(l_j) \\
0 \leqslant l_j \leqslant n_j}}(-1)^{|\mathfrak{l}|}
\binom{\mathbf{k}_1+\mathbf{n}-\mathbf{1}}{\mathbf{n}-\mathfrak{l}}
\binom{\mathbf{k}_2+\mathbf{n}-\mathbf{1}}{\mathfrak{l}}
 f_1^{(\mathfrak{l})}(z) f_2^{(\mathbf{n}-\mathfrak{l})}(z),$$
where $|\mathfrak{l}|=\sum_{j}l_j$. Then for all $M \in \mathrm{SL}_2(\mathbb{R})^d$,
$$\left[\left(\left.f_1\right|_{\mathbf{k}_1} M\right),\left(\left.f_2\right|_{\mathbf{k}_2} M\right)\right]_{\mathbf{n}}=\left.\left[f_1, f_2\right]_{\mathbf{n}}\right|_{\mathbf{k}_1+\mathbf{k}_2+2 \mathbf{n}} M.$$
In particular, if $\Gamma$ is a subgroup of $\mathrm{SL}_2(\mathbb{R})^d$ and $f_r \in M_{\mathbf{k}_r}(\Gamma)$ for $r=1,2$, then $\left[f_1, f_2\right]_{\mathbf{n}} \in M_{\mathbf{k}_1+\mathbf{k}_2+2 \mathbf{n}}(\Gamma)$.

\begin{theorem}\label{poincare series rankin cohen}
Let $\mathbf{n}\in \mathbb{N}_0^d, \mathbf{k}_1 , \mathbf{k}_2 \in \mathbb{Z}_{>2}^d$ and  $\mu\in\mathfrak{c}^+$. For any $f \in M_{\mathbf{k}_1}$, we assume further $\mathbf{k}_2 \geq \mathbf{k}_1+2$ if $f$ is not a cusp form. Then $[f, P_{\mu}(z;\mathbf{k}_2,\mathfrak{c},\mathfrak{n})]_{\mathbf{n}}=\mathbb{P}_{\mathbf{k}_1+\mathbf{k}_2+2\mathbf{n}}(\phi)$, where
$$\phi(z)=e(\mu z)\sum_{\substack{\mathfrak{l} \\ 
0 \leqslant l_j \leqslant n_j}}(-1)^{|\mathfrak{l}|}
\binom{\mathbf{k}_1+\mathbf{n}-\mathbf{1}}{\mathbf{n}-\mathfrak{l}}
\binom{\mathbf{k}_2+\mathbf{n}-\mathbf{1}}{\mathfrak{l}}
 N(2\pi i\mu)^{\mathbf{n}-\mathfrak{l}} f^{(\mathfrak{l})}(z).$$
\end{theorem}
\begin{proof}
The proof proceeds the same as in the elliptic case \cite[Theorem 2]{W18}, so we omit it here.
\end{proof}
Since Rankin-Cohen brackets are bilinear, we see that the Rankin-Cohen brackets and Poincar\'e averaging `commute' in the following sense:
\begin{cor}
Let $f$ be a Hilbert modular form of weight $\mathbf{k}$ and $\phi$ be a q-series whose coefficients have upper bound $O\left(N(\nu)^{\frac{\mathfrak{l}}{2}-2-\varepsilon}\right)$ for some $\varepsilon>0$. If $f$ is not a cusp form, we assume further $\mathfrak{l} \geq \mathbf{k}+2$. Then
$$\left[f, \mathbb{P}_{\mathfrak{l}}(\phi)\right]_{\mathbf{n}}=\mathbb{P}_{\mathbf{k}+\mathfrak{l}+2\mathbf{n}}\left([f, \phi]_{\mathbf{n}}\right) .$$
Here we denote formally $[f, \phi]_{\mathbf{n}}$ as the $\mathbf{n}$-th Rankin-Cohen bracket of $f$ and $\phi$, where $\phi$ is treated like a modular form of weight $\mathfrak{l}$.
\end{cor}

This expression simplifies considerably for the Eisenstein series to be $\left[f, E_{\mathfrak{l}}\right]_{\mathbf{n}}=\mathbb{P}_{\mathbf{k}+\mathfrak{l}+2\mathbf{n}}(\phi)$ with $\phi=(-1)^{|\mathbf{n}|}
\binom{\mathfrak{l}+\mathbf{n}-\mathbf{1}}{\mathbf{n}}
 f^{(\mathbf{n})}$.

\begin{cor}\label{Eisenstein series rankin cohen}
Let $E_{\mathbf{k}}(z;\mathfrak{c},\mathfrak{n})$ be the Hilbert Eisenstein series and $g_{\mathfrak{l}}(z)=\sum_{\mu \in \mathfrak{c}^+ \cup \{ 0 \}} b_{\nu} e(\nu z)  \in M_{\mathfrak{l}}(\Gamma)$. Then, for all $\mathbf{n} \in \mathbb{N}^{d}$, we have
$$\left[E_{\mathbf{k}}(z;\mathfrak{c},\mathfrak{n}), g_{\mathfrak{l}}\right]_{\mathbf{n}}=(-1)^{|\mathbf{n}|}
\binom{\mathbf{k}+\mathbf{n}-\mathbf{1}}{\mathbf{n}}
 \sum_{\nu \in \mathfrak{c}^+\slash\mathcal{O}^{\times +}}(2 \pi i \nu)^{\mathbf{n}} b_{\nu} P_{\nu}(z;\mathbf{k}+\mathfrak{l}+2\mathbf{n},\mathfrak{c},\mathfrak{n}).$$
\end{cor}

\begin{prop}
Let $\mathbf{m}=\mathbf{k}+\mathfrak{l}+2 \mathbf{n}$ with $\mathbf{k},\mathfrak{l}>2$ and $\mathbf{n} \in \mathbb{N}^d$ be parallel weights. Suppose that $f_{\mathbf{m}}(z)=\sum_{\nu \in \mathfrak{c}^+} a_{\nu} e(\nu z) \in S_{\mathbf{m}}(\Gamma)$ and
$g_{\mathfrak{l}}(z)=\sum_{\mu \in \mathfrak{c}^+ \cup \{ 0 \}} b_{\mu} e(\mu z) \in M_{\mathfrak{l}}(\Gamma)$.
Denote $\mathfrak{m}=\nu\mathfrak{c}^{-1}$ and $c(\mathfrak{m},f)=N(\mathfrak{c})^{-\frac{k}{2}}a_\nu$. Then
$\left\langle f_{\mathbf{m}},\left[E_{\mathbf{k}}(z;\mathfrak{c},\mathfrak{n}), g_{\mathfrak{l}}\right]_{\mathbf{n}}\right\rangle$ equals 
$$\binom{\mathbf{k}+\mathbf{n}-\mathbf{1}}{\mathbf{n}}
\frac{(-2\pi i)^{\mathbf{n}}  \sqrt{D}\left((4 \pi)^{1-m} \Gamma(m-1)\right)^{d}}{N(\mathfrak{c})^{\frac{k}{2}}N(\mathfrak{d})}
\sum_{\nu \in \mathfrak{c}^+\slash\mathcal{O}^{\times +}}
c(\mathfrak{m},f)\overline{c(\mathfrak{m},g)}N(\mathfrak{m})^{1+n-m}.$$
\end{prop}
\begin{proof}
By Corollary \ref{Eisenstein series rankin cohen}, $\left\langle f_{\mathbf{m}},\left[E_{\mathbf{k}}(z;\mathfrak{c},\mathfrak{n}), g_{\mathfrak{l}}\right]_{\mathbf{n}}\right\rangle$ equals
\begin{align*}
&\left\langle f_{\mathbf{m}},\binom{\mathbf{k}+\mathbf{n}-\mathbf{1}}{\mathbf{n}} \sum_{\nu \in \mathfrak{c}^+\slash\mathcal{O}^{\times +}}(2 \pi i \nu)^{\mathbf{n}} b_{\nu} P_{\nu}(z;\mathbf{k}+\mathfrak{l}+2\mathbf{n},\mathfrak{c},\mathfrak{n})\right\rangle\\
&= \binom{\mathbf{k}+\mathbf{n}-\mathbf{1}}{\mathbf{n}} \sum_{\nu \in \mathfrak{c}^+\slash\mathcal{O}^{\times +}}(-2 \pi i \nu)^{\mathbf{n}} \bar{b_{\nu}}
\left\langle f_{\mathbf{m}}, P_{\nu}(z;\mathbf{k}+\mathfrak{l}+2\mathbf{n},\mathfrak{c},\mathfrak{n})\right\rangle\\
&=\binom{\mathbf{k}+\mathbf{n}-\mathbf{1}}{\mathbf{n}}
\sum_{\nu \in \mathfrak{c}^+\slash\mathcal{O}^{\times +}}(-2 \pi i \nu)^{\mathbf{n}} \bar{b_{\nu}}
a_\nu N(\mathfrak{cd})^{-1} N(\nu)^{1-m} \sqrt{D}\left((4 \pi)^{1-m} \Gamma(m-1)\right)^{d}\\
&= \binom{\mathbf{k}+\mathbf{n}-\mathbf{1}}{\mathbf{n}}
(-2\pi i)^{\mathbf{n}} N(\mathfrak{cd})^{-1} \sqrt{D}\left((4 \pi)^{1-m} \Gamma(m-1)\right)^{d}
\sum_{\nu \in \mathfrak{c}^+\slash\mathcal{O}^{\times +}} \bar{b_{\nu}} a_\nu N(\nu)^{1+n-m},
\end{align*}
where the last summation equals $N(\mathfrak{c})^{\frac{2n+l-m}{2}+1} \sum_{\nu \in \mathfrak{c}^+\slash\mathcal{O}^{\times +}} 
c(\mathfrak{m},f)\overline{c(\mathfrak{m},g)}N(\mathfrak{m})^{1+n-m}$.
\end{proof}

Now we can relate the Petersson inner product of a cuspidal Hilbert modular form with the Rankin-Cohen product of Eisenstein series and another ad\'elic Hilbert modular form to a special value of the Rankin-Selberg $L$-function associated to them. 

\begin{cor}\label{Rankin Selberg}
Let $\mathbf{m}=\mathbf{k}+\mathfrak{l}+2 \mathbf{n}$ with $\mathbf{k}, \mathfrak{l}>2$ and $\mathbf{n} \in \mathbb{N}^{d}$ be parallel weights. Suppose $f(z) \in \mathcal{S}_\mathbf{m}(\mathfrak{n})$ and
$g(z)\in \mathcal{M}_{\mathfrak{l}}(\mathfrak{n})$. Denote $E_\mathbf{k}=(N(\mathfrak{c}_j)^{\frac{k}{2}} E_{k}(z;\mathfrak{c}_j,\mathfrak{n}))_j$. Then
$$\left\langle f,\left[E_{\mathbf{k}}, g\right]_{\mathbf{n}}\right\rangle
=\genfrac(){0pt}{1}{\mathbf{k}+\mathbf{n}-\mathbf{1}}{\mathbf{n}}(-2\pi i)^{\mathbf{n}} N(\mathfrak{d})^{-1} \sqrt{D}\left((4 \pi)^{1-m} \Gamma(m-1)\right)^{d}
L(f\times g,m-n-1).$$
\end{cor}
\begin{proof}
Since $\langle f, g\rangle=\sum_{j}\left\langle f_{j}, g_{j}\right\rangle_{\Gamma_{j}}$, the statement follows immediately from the preceding proposition.
\end{proof}

Shimura constructed an Eisenstein series $\tilde{E_\mathfrak{l}}$ for parallel weight $\mathfrak{l}$ such that \[c\left(\mathfrak{m}, \tilde{E_{\mathfrak{l}}}\right)=\sum_{\mathfrak{r} \mid \mathfrak{m}}  \mathrm{N}\mathfrak{r}^{l-1}=\sigma_{l-1}(\mathfrak{m})\] for all nonzero integral ideals $\mathfrak{m}$ \cite{S78}. Later, Dasgupta, Darmon and Pollack \cite[Proposition 2.1]{Das11} wrote $\tilde{E_\mathfrak{l}}=(f_j)_j$, where
$$f_j(z)=C \sum_{\mathcal{C}\in Cl(F)} N(\mathcal{C})^l \sum_{\substack{a\in\mathcal{C}\\b\in \mathfrak{d}^{-1}\mathfrak{c}_j^{-1}\mathcal{C}}}N(az+b)^{-l}
=C\sum_{\mathcal{C}\in Cl(F)}N(\mathcal{C})^l \sum_{a\in\mathcal{O}/\mathcal{C}} E_\mathfrak{l}(z+a;\mathfrak{c}_j,\mathfrak{n})$$
in our setting. Here $C$ is a constant depending only on $k$ and $F$. When the class number is one, these two kinds of Eisenstein series differ by a constant.

\begin{theorem}\label{Zagier's kernel formula}%with all necesary notations in the statement of theorem
Let $k, l \geq 4$ be even integers, $n$ be a positive integer and $m=k+l+2 n$. Suppose $\mathfrak{n}=\mathcal{O}$ and that  $f \in \mathcal{S}_\mathbf{m}(\mathcal{O})$ is a normalized eigenform. Then $\langle f,[E_{\mathbf{k}}, \tilde{E_{\mathfrak{l}}}]_{\mathbf{n}}\rangle $ equals 
$$\genfrac(){0pt}{1}{\mathbf{k}+\mathbf{n}-\mathbf{1}}{\mathbf{n}}
(-2\pi i)^{\mathbf{n}}  \sqrt{D}\left((4 \pi)^{1-m} \Gamma(m-1)\right)^{d} N(\mathfrak{d})^{-1}\zeta_F(k)^{-1}
L(f,m-n-1 ) L\left(f, m-n-l\right).$$
\end{theorem}
\begin{proof}
By Corollary \ref{Rankin Selberg}, it suffices to show that for an eigenform $f$,
$$L(f\times \tilde{E_\mathfrak{l}} ,s)=\sum_{\mathfrak{m}\subset \mathcal{O}} \frac{\sigma_{l-1}(\mathfrak{m})c(\mathfrak{m},f)}{N(\mathfrak{m})^s}
=\frac{L(f,s)L(f,s-l+1)}{\zeta_F(2s-l-m+2)}.$$
Since $c(\mathfrak{m},f)c(\mathfrak{n},f)=\sum_{(\mathfrak{m},\mathfrak{n})\subset \mathfrak{a}} N(\mathfrak{a})^{m-1} c(\mathfrak{mna}^{-2},f)$, we have
\begin{align*}
& \zeta_F (2s-l-m+2)^{-1} \sum_{\mathfrak{m}}\sum_{\mathfrak{n}}\frac{c(\mathfrak{m},f)}{N(\mathfrak{m})^{s}}\frac{c(\mathfrak{n},f)}{N(\mathfrak{n})^{s-l+1}}\\
&= \zeta_F (2s-l-m+2)^{-1} \sum_{\mathfrak{m}}\sum_{\mathfrak{n}}\sum_{(\mathfrak{m},\mathfrak{n})\subset \mathfrak{a}}\frac{N(\mathfrak{a})^{m-1} c(\mathfrak{mna}^{-2},f)}{N(\mathfrak{m})^{s}N(\mathfrak{n})^{s-l+1}}\\
&= \zeta_F (2s-l-m+2)^{-1} \sum_{\tilde{\mathfrak{m}}}c(\tilde{\mathfrak{m}},f) \sum_{\tilde{\mathfrak{n}}}
\sum_{\tilde{\mathfrak{a}} \mid \tilde{\mathfrak{m}}} \frac{N(\tilde{\mathfrak{n}})^{m-1}}{N(\tilde{\mathfrak{a}})^{s-l+1}N(\tilde{\mathfrak{n}})^{s-l+1}N(\tilde{\mathfrak{m}}\tilde{\mathfrak{a}}^{-1})^sN(\tilde{\mathfrak{n}})^s}\\
&= \zeta_F (2s-l-m+2)^{-1} \sum_{\tilde{\mathfrak{m}}}\frac{c(\tilde{\mathfrak{m}},f)\sigma_{l-1}(\tilde{\mathfrak{m}})}{N(\tilde{\mathfrak{m}})^s}  \zeta_F (2s-l-m+2),
\end{align*}
where $\mathfrak{n}=\tilde{\mathfrak{a}}\tilde{\mathfrak{n}}$, $\mathfrak{m}=\tilde{\mathfrak{m}}\tilde{\mathfrak{n}}\tilde{\mathfrak{a}}^{-1}$ and $\mathfrak{a}=\tilde{\mathfrak{n}}$. This finishes the proof. 
\end{proof}

\section{Kernels for products of Hilbert $L$-functions}
Throughout this section, we shall assume all weights are even parallel and bigger than 2. In order to use Hilbert Poincar\'e series to construct Cohen's kernel and double Eisenstein series, we first calculate its upper bound. For an ad\'elic Hilbert modular form $f=(f_j)_j$, we denote $|f(z)|=\sum_j |f_j(z)|$. 

As in \cite{ZZ}, we define the ad\'elic Hilbert Poincar\'e series associated to a nonzero integral ideal $\mathfrak{m}$ as follows. If $\mathfrak{m}=(\mu)\mathfrak{c}_j^{-1}$, $\mu \gg 0$, then $\mu\in\mathfrak{c}_j^{+}$ and define $P_{\mathfrak{m}}(z;k,\mathfrak{n})$ to be the ad\'elic Hilbert modular form whose $j$-th component is equal to
$$\frac{N(\mathfrak{c}_j)^{\frac{k+2}{2}}D^{1/2}N(\mu)^{k-1}}{(4\pi)^{(1-k)d}\Gamma(k-1)^d} P_{\mu}(z;k,\mathfrak{c}_j,\mathfrak{n})$$
and other components are $0$. Notice that
$$\langle f,P_{\mathfrak{m}}(z;k,\mathfrak{n})\rangle =\langle f_j,P_{\mathfrak{m}}(z;k,\mathfrak{n})_j\rangle_{\Gamma_j}=c(\mathfrak{m}, f),$$
so $\{ P_{\mathfrak{m}}(z;k,\mathfrak{n}) | \mathfrak{m} \text{ is an integral ideal}\}$ spans $\mathcal{S}_k (\mathfrak{n})$ as expected.

\begin{lemma}\label{upper poincare}
For $z$ in any compact subset $K$ of $\mathbb{H}^d$, $|P_{\mathfrak{m}}(z)|\leq C N(\mathfrak{m})^{k-\frac{3}{4}}$ with $C$ a constant depending only on $k,F$ and $K$. Consequently, the series $\sum_{\mathfrak{m}}N(\mathfrak{m})^s P_{\mathfrak{m}}(z)$ is absolutely and uniformly convergent on $K$ when $Re(s)<-k-\frac{1}{4}$.
\end{lemma}
\begin{proof}
For $\nu\in\mathfrak{c}^+$, the $\nu$-th coefficient $c_k(\nu,\mu)=c_k(\nu,\mu;\mathfrak{c},\mathfrak{n})$ of $P_{\mu}(z;k, \mathfrak{c},\mathfrak{n})$ equals 
\begin{align*}
\chi_{\mu}(\nu)+ \left(\frac{N(\nu)}{N(\mu)}\right)^{\frac{k-1}{2}}\frac{(2\pi)^d(-1)^{dk/2}N(\mathfrak{cd})}{D^{1 / 2}}
\sum_{\varepsilon\in \mathcal{O}^{\times +}}\sum_{c\in(\mathfrak{cnd}/\mathcal{O}^{\times})^{\ast}}\frac{ S_{c(\mathfrak{cd})^{-1}}(\nu, \varepsilon\mu  ; c)}{|N(c)|}NJ_{k-1} \left(\frac{4 \pi \sqrt{\nu\varepsilon\mu}}{|c|}\right).
\end{align*}
We can choose $\varepsilon^\prime\in\mathcal{O}^{\times +}$ such that $A^{-1} |N(\nu\mu)|^{\frac{1}{d}}\leq |(\nu\mu\varepsilon^\prime)_i|\leq A |N(\nu\mu)|^{\frac{1}{d}}$, where $A$ is a constant depending only on $F$. Then
\begin{align*}
c_k(\nu,\mu)&=\chi_{\mu}(\nu)+ \left(\frac{N(\nu)}{N(\mu)}\right)^{\frac{k-1}{2}}\frac{(2\pi)^d(-1)^{dk/2}N(\mathfrak{cd})}{D^{1 / 2}}\\
&\qquad \times\sum_{\varepsilon\in \mathcal{O}^{\times +}}\sum_{c\in(\mathfrak{cnd}/\mathcal{O}^{\times})^{\ast}}\frac{ S_{c(\mathfrak{cd})^{-1}}(\nu, \varepsilon\varepsilon^\prime\mu  ; c)}{|N(c)|}NJ_{k-1} \left(\frac{4 \pi \sqrt{\nu\varepsilon\varepsilon^\prime\mu}}{|c|}\right).
\end{align*}
Since
$$\left| S_{c(\mathfrak{cd})^{-1}}(\nu, \varepsilon \mu  ; c) \right| \leq  C_1\frac{N(\nu\mu)^{\frac{1}{4}}|N(c)|}{N(\mathfrak{cd})},$$
and
$$ NJ_{k-1} \left(\frac{4 \pi \sqrt{\varepsilon\varepsilon^\prime\nu\mu}}{|c|}\right) \leq C_2(\frac{4\pi e}{2k-2})^{dk-d}N(\nu\mu)^{\frac{k-1}{2}}|N(c)|^{-k+1+\eta} \prod_{\left|\varepsilon_j\right|>1}\left|\varepsilon_j\right|^{-\eta}$$
for any $0<\eta <1$ with $C_1, C_2$ depending on $F$ and $\eta$ only as in \cite{ZZ}, we have
$$|c_k(\nu,\mu)|\leq 1+ C_3 N(\nu)^{k-\frac{3}{4}}N(\mu)^{\frac{1}{4}}$$
with $C_3$ depending on $F$ and $k$ only. Hence for all $z\in K$, we have
$$|P_{\mathfrak{m}}(z)|\leq C_4 N(\mathfrak{m})^{k-1}|P_\mu(z)| \leq C_5 N(\mathfrak{m})^{k-1} \sum_{\nu\in\mathfrak{c}_j^+/\mathcal{O}^{\times +}} N(\nu)^{k-\frac{3}{4}}N(\mu)^{\frac{1}{4}}\sum_{\varepsilon\in\mathcal{O}^{\times +}} e^{-2\pi\nu\varepsilon y}$$
with $C_4,C_5$ depending on $F$ and $k$ only. We may also assume that $\nu$ satisfy $A^{-1}N(\nu)^{\frac{1}{d}}\leq \nu_i\leq AN(\nu)^{\frac{1}{d}}$ for each $i$. Then
$$\sum_{\varepsilon\in\mathcal{O}^{\times +}} e^{-2\pi\varepsilon\nu y}\leq (2m!)^d\sum_{\varepsilon\in\mathcal{O}^{\times +}} \prod_{\varepsilon_j >1}\frac{1}{(2\pi \varepsilon_j\nu_j y_j)^{2m}}\prod_{\varepsilon_j \leq 1}\frac{1}{(2\pi \varepsilon_j\nu_j y_j)^{m}}\leq \frac{C_6}{N(\nu)^{\frac{3m}{d}}}\sum_{\varepsilon\in\mathcal{O}^{\times +}} \prod_{\varepsilon_j >1}\varepsilon_j^{-m}$$
for any $m>0$ with $C_6$ depending on $m$ and $F$, where the last equation is convergent as in \cite{Luo03}. Hence $|P_{\mathfrak{m}}(z)|\leq C N(\mathfrak{m})^{k-\frac{3}{4}}$ with $C$ depending on $k,F$ and $K$ only.
\end{proof}

\subsection{Cohen's kernel}
Kernel of $L$-function was first considered by Cohen \cite{C81} and is now known as Cohen's kernel, which is a cusp form and was expressed as series of Poincar\'e series by Zagier in \cite{Z77}. In this subsection, we obtain similar expression over totally real number fields, which will be used in next subsection to construct double Eisenstein series.

We define the $i$-th Cohen's kernel $\mathcal{C}_k^{i}(z ; s)$ on $\mathbb{H}^d \times \mathbb{C}$ by
$$\mathcal{C}_k^i(z ; s)= \frac{e^{\pi i s} \Gamma(s)^d (4 \pi)^{(k-1) d} }{(2 \pi)^{d s}N(\mathfrak{d})^{1 / 2}\Gamma(k-1)^d}
\sum_{\gamma \in U(\Gamma_i)Z(\Gamma_i) \backslash \Gamma_i}(\gamma z)^{-s \mathbf{1}} j(\gamma, z)^{-k \mathbf{1}}$$
with $\Gamma_i=\Gamma_0(\mathfrak{c}_i,\mathfrak{n})$. Note that if $k$ is odd, this definition gives zero function. Then we define Cohen's kernel to be $\mathcal{C}_k(z ; s)=(N(\mathfrak{c}_i)^{\frac{3k}{2}-s}\mathcal{C}_k^{i}(z ; s))_i$.

\begin{prop}\label{Cohen kernel Poincare series}
$\mathcal{C}_k(z ; s)$ converges absolutely and uniformly on all compact subsets in 
$$\{ (z,s)\in \mathbb{H}^d\times \mathbb{C} | Re(s)<-\frac{1}{4}\},$$
and $\mathcal{C}_k(z ; s)=\sum_{\mathfrak{m}} N(\mathfrak{m})^{s-k}P_{\mathfrak{m}}(z;k,\mathfrak{n})$ is a cusp form in $\mathcal{S}_k(\mathfrak{n})$. Consequently, 
$$\mathcal{C}_k(z ; s)=\sum_{f \in \mathcal{H}_k} \frac{L(f, k-s)}{\langle f, f\rangle} f(z),$$
where $\mathcal{H}_k$ is the set of primitive forms of weight $k$ in $\mathcal{S}_k(\mathfrak{n})$.
\end{prop}
\begin{proof}
By the multi-variable Lipschitz summation formula, we have 
\begin{align*}
& \sum_{\gamma \in U(\Gamma_i)Z(\Gamma_i) \backslash \Gamma_i}(\gamma z)^{-s \mathbf{1}} j(\gamma, z)^{-k \mathbf{1}}
=\sum_{\gamma\in U(\Gamma_i)Z(\Gamma_i)N(\Gamma_i)\backslash \Gamma}j(\gamma, z)^{-k \mathbf{1}}
\sum_{x\in(\mathfrak{c}_i\mathfrak{d})^{-1}} (\gamma z+x)^{-s \mathbf{1}}\\
&= \frac{(2 \pi)^{d s}D^{1 / 2}N(\mathfrak{d})^{1 / 2}N(\mathfrak{c})}{e^{\pi i s} \Gamma(s)^d } 
\sum_{\gamma\in U(\Gamma_i)Z(\Gamma_i)N(\Gamma_i)\backslash \Gamma}j(\gamma, z)^{-k \mathbf{1}}
\sum_{\xi\in\mathfrak{c}_i^+} N(\xi)^{s-1}e(\xi \gamma z)\\
&= \frac{(2 \pi)^{d s}D^{1 / 2}N(\mathfrak{d})^{1 / 2}N(\mathfrak{c})}{e^{\pi i s} \Gamma(s)^d }  \sum_{\xi\in\mathfrak{c}_i^+} N(\xi)^{s-1}
\sum_{\gamma\in U(\Gamma_i)Z(\Gamma_i)N(\Gamma_i)\backslash \Gamma} j(\gamma, z)^{-k \mathbf{1}} e(\xi \gamma z)\\
&= \frac{(2 \pi)^{d s}D^{1 / 2}N(\mathfrak{d})^{1 / 2}N(\mathfrak{c})}{e^{\pi i s} \Gamma(s)^d } 
\sum_{\xi\in\mathfrak{c}_i^+\slash \mathcal{O}^{\times}_+} N(\xi)^{s-1}P_\xi(z;k,\mathfrak{c}_i,\mathfrak{n}),
\end{align*}
which converges absolutely and uniformly on all compact subsets in $\{ (z,s)\in \mathbb{H}^d\times \mathbb{C} | Re(s)<-\frac{1}{4}\}$ by Lemma \ref{upper poincare}. Now the first statement follows from the definition of ad\'elic Hilbert Poincar\'e series. By comparing the Petersson inner product with ad\'elic Hilbert modular forms, we have
$$\sum_{\mathfrak{m}} N(\mathfrak{m})^{s-k}P_{\mathfrak{m}}(z;k,\mathfrak{n})=\sum_{f \in \mathcal{H}_k} \frac{L(f,k-s)}{\langle f, f\rangle}f(z)$$ 
in the convergent region. 
\end{proof}

The Cohen's kernel has analytic continuation via
$$\mathcal{C}_k^\ast (z ; s)=\sum_{f \in \mathcal{H}_k} \frac{\Lambda(f,k-s)}{\langle f, f\rangle}f(z)
=\sum_{f \in \mathcal{H}_k} \frac{\Lambda(f,s)}{\langle f, f\rangle}f(z),$$
where $\Lambda(f, s)=D^s(2 \pi)^{-d s} \Gamma(s)^d L(f, s)=(-1)^k \Lambda(f, k-s)$, for all $s\in\mathbb{C}$.

\subsection{Double Eisenstein series}
In this subsection, we express the kernel for products of $L$-functions
$$E_{s, k-s}(\cdot , w) :=\sum_{f \in \mathcal{H}_k} \frac{L(f, k-s) L(f, k-w)}{\langle f, f\rangle} f$$ 
as a series of Hilbert Poincar\'e series. Although this definition differs from the one defined in Diamantis and O'Sullivan in \cite{DO13} by a factor containing $s$ and $w$, we still call it double Eisenstein series because they possess similar properties. We define the completed double Eisenstein series
$E_{s, k-s}^{*}(\cdot , w) :=\sum_{f \in \mathcal{H}_k} \frac{\Lambda(f, s) \Lambda(f, w)}{\langle f, f\rangle} f$ as in \cite{DO13}.

\begin{prop} In the region $\{ (s,w)\in\mathbb{C}^2 | Re(s),Re(w)<-\frac{5}{4} \}$, we have
$$E_{s, k-s}(\cdot , w) = \zeta_F (-s-w+k+1) \sum_{\mathfrak{m}}N(\mathfrak{m})^{s-k}\sigma_{w-s}(\mathfrak{m})P_{\mathfrak{m}}.$$
\end{prop}
\begin{proof}
For any $s,t\in\mathbb{C}$, by a similar computation as in Theorem \ref{Zagier's kernel formula},
\begin{align*}
\sum_{\mathfrak{m},\mathfrak{n}} N(\mathfrak{m})^s N(\mathfrak{n})^t T_{\mathfrak{n}}P_{\mathfrak{m}}
&= \sum_{\mathfrak{m},\mathfrak{n}} N(\mathfrak{m})^s N(\mathfrak{n})^t \sum_{\mathfrak{m}+\mathfrak{n}\subset\mathfrak{a}} N(\mathfrak{a})^{k-1}P_{\mathfrak{mna}^{-2}}
=\frac{\sum_{\tilde{\mathfrak{m}}}N(\tilde{\mathfrak{m}})^{s}\sigma_{t-s}(\tilde{\mathfrak{m}})P_{\tilde{\mathfrak{m}}}}{\zeta_F (-s-t-k+1) ^{-1}}.
\end{align*}
Hence
\begin{align*}
&E_{s, k-s}(\cdot , w) = \sum_{f \in \mathcal{H}_k} \sum_{\mathfrak{n}} \frac{L(f, k-s) c(\mathfrak{n},f)}{\langle f, f\rangle N(\mathfrak{n})^{k-w}} f = \sum_{\mathfrak{n}} N(\mathfrak{n})^{w-k} T_{\mathfrak{n}}\left(\mathcal{C}_k (z ; s)\right)\\
&= \sum_{\mathfrak{n}} N(\mathfrak{n})^{w-k} T_{\mathfrak{n}}\left(
\sum_{\mathfrak{m}} N(\mathfrak{m})^{s-k}P_{\mathfrak{m}}(z;k,\mathfrak{n})\right)
=\frac{ \sum_{\tilde{\mathfrak{m}}}N(\tilde{\mathfrak{m}})^{s-k}\sigma_{w-s}(\tilde{\mathfrak{m}})P_{\tilde{\mathfrak{m}}}}
{\zeta_F (-s-w+k+1) ^{-1}},
\end{align*}
which is absolutely convergent in the region $\{ (s,w)\in\mathbb{C}^2 | Re(s),Re(w)<-\frac{5}{4} \}$ by Lemma \ref{upper poincare} using the bound
$$\sigma_r(\mathfrak{m})\leq\begin{cases}
N(\mathfrak{m}),\quad &Re(r)\leq 0\\
N(\mathfrak{m})^{r+1},\quad & Re(r)>0\end{cases}.$$
\end{proof}

In the convergent region, the $i$-th component of double Eisenstein series is
$$ \zeta_F (-s-w+k-1) \sum_{\mathfrak{m}\sim \mathfrak{c}_i}N(\mathfrak{m})^{s-k}\sigma_{w-s}(\mathfrak{m})P_{\mathfrak{m},i},$$
where the summation is taken over all integral ideals in the narrow class $\mathfrak{c}_i$ and $P_{\mathfrak{m},i}$ is the $i$-th component of $P_{\mathfrak{m}}$.

\begin{theorem}\label{rcb double}
Suppose $\mathfrak{n}=\mathcal{O}$. Let $k,l\in\mathbb{Z}_{>2}$ be even and $n\in\mathbb{N}$. Denote $m=k+l+2n$. Then
$$\left[E_{\mathbf{k}}, \tilde{E_{\mathfrak{l}}}\right]_{\mathbf{n}}=
\binom{\mathbf{k}+\mathbf{n}-\mathbf{1}}{\mathbf{n}}(-2\pi i)^{\mathbf{n}}  \sqrt{D}\left((4 \pi)^{1-m} \Gamma(m-1)\right)^{d} N(\mathfrak{d})^{-1}\zeta_F(k)^{-1}
E_{\mathfrak{l}+\mathbf{n},\mathbf{k}+\mathbf{n}}(\cdot,\mathbf{n}+1).$$
\end{theorem}
\begin{proof}
This follows immediately from Corollary \ref{Zagier's kernel formula}.
\end{proof}

Now we recover Shimura's algebraicity result of Hilbert $L$-values using the method as in \cite[Theorem 2.5]{CZ20}. For completeness, we sketch the proof here.

\begin{cor}\label{rationality}
Suppose $h^+=1$ and $f$ be a primitive form of even parallel weight $k \geq 6$ for $\Gamma_0(\mathcal{O})$. Then there exist complex numbers $\omega_{ \pm}(f)$ with $\langle f, f\rangle=\omega_{+}(f) \omega_{-}(f)$ such that for even $m$ and odd $\ell$ with $1 \leq m, \ell \leq k-1$,
$$\frac{\Lambda(f, m)}{w_{+}(f)}, \frac{\Lambda(f, \ell)}{w_{-}(f)} \in \mathbb{Q}(f),$$
and for each $\sigma \in \operatorname{Gal}(\overline{\mathbb{Q}} / \mathbb{Q})$,
$$\left(\frac{\Lambda(f, m)}{\omega_{+}(f)}\right)^\sigma=\frac{\Lambda\left(f^\sigma, m\right)}{\omega_{+}\left(f^\sigma\right)}, \quad\left(\frac{\Lambda(f, \ell)}{\omega_{-}(f)}\right)^\sigma=\frac{\Lambda\left(f^\sigma, \ell\right)}{\omega_{-}\left(f^\sigma\right)}.$$
\end{cor}
\begin{proof}
We first show that $E^\ast_{m,k-m}(\cdot,l)$ has rational coefficients. By the functional equation of $\Lambda(f,s)$, it suffices to consider $1\leq l < m \leq \frac{k}{2}$. Theorem \ref{rcb double} tells us that $E^\ast_{m,k-m}(\cdot,l)$ equals $[E_{\mathbf{k-m}-\mathfrak{l}+1}, \tilde{E}_{\mathfrak{m-l}+1}]_{\mathfrak{l}-1}$ up to a constant. By substituting the Fourier expansion of Eisenstein series, $E^\ast_{m,k-m}(\cdot,l)$ have rational coefficients from \cite[Corollary 9.9]{N99} and the functional equation of $\zeta_F$. Similarly, $E^\ast_{k-1,2}(z ; k-1)$ also have rational coefficients. Then for primitive $f \in \mathcal{H}_k$, we have $\left\langle f, E^\ast_{k-1,2}(z ; k-1)\right\rangle=\alpha_f\langle f, f\rangle=\Lambda(f, k-1) \Lambda(f, k-2)$ for certain $\alpha_f \in \mathbb{Q}(f)$ by \cite[Proposition 4.15]{S78}. Define
$$\omega_{+}(f)=\frac{\alpha_f\langle f, f\rangle}{\Lambda(f, k-1)}, \quad \omega_{-}(f)=\frac{\langle f, f\rangle}{\Lambda(f, k-2)} .$$
Then $\omega_{+}(f) \omega_{-}(f)=$ $\langle f, f\rangle$. For even $m$, odd $\ell$ with $1 \leq m, \ell \leq k-1$, 
$\frac{\Lambda(f, m)}{\omega_{+}(f)}=\frac{\left\langle f, E_{m, k-m}(z ; k-1)\right\rangle}{\alpha_f\langle f, f\rangle} \in \mathbb{Q}(f)$
and similarly $\frac{\Lambda(f, \ell)}{\omega_{-}(f)} \in \mathbb{Q}(f)$.  Finally, the assertion (4.16) of \cite{S78} and that $\alpha_f^\sigma=\alpha_{f^\sigma}$ for each $\sigma \in \operatorname{Gal}(\overline{\mathbb{Q}} / \mathbb{Q})$ implies that
$$\left(\frac{\Lambda(f, m)}{\omega_{+}(f)}\right)^\sigma
=\left( \frac{\left\langle f, E_{m, k-m}(z ; k-1)\right\rangle}{\alpha_f\langle f, f\rangle} \right)^\sigma
=\frac{\Lambda\left(f^\sigma, m\right)}{\omega_{+}\left(f^\sigma\right)},$$
and similarly $\left(\frac{\Lambda(f, \ell)}{\omega_{-}(f)}\right)^\sigma
=\frac{\Lambda\left(f^\sigma, \ell\right)}{\omega_{-}\left(f^\sigma\right)}$.
\end{proof}

\noindent
\textbf{Acknowledgements } The authors are supported by a grant of National Natural Science Foundation of China (no. 12271123). The first named author would like to thank China Scholarship Council (no. 202206120153) for supporting him to visit TU Darmstadt as a joint-PhD student during the preparation of this paper.

\end{document}